\date{July 21, 2014}
\documentclass[a4paper,10pt,reqno]{amsart}
\usepackage{latexsym,amsmath,amsfonts,amscd,amssymb}
\usepackage{graphicx}
\usepackage{eucal}
\usepackage[all]{xy}
\usepackage{booktabs}
\usepackage{color}
\definecolor{Myblue}{rgb}{0,0,0.6}
\usepackage[colorlinks,citecolor=Myblue,linkcolor=Myblue,urlcolor=Myblue,pdfpagemode=UseNone,backref = page]{hyperref}

\newcommand\bC{\mathbb{C}}

\newcommand\bR{\mathbb{R}}
\newcommand\bZ{\mathbb{Z}}

\newcommand\fg{\mathfrak{g}}
\newcommand\fh{\mathfrak{h}}

\newcommand\fX{\mathfrak{X}}

\newtheorem{theorem}{Theorem}

\newtheorem{proposition}[theorem]{Proposition}

\newtheorem{corollary}[theorem]{Corollary}
\newtheorem{remark}[theorem]{Remark}
\theoremstyle{definition}\newtheorem{ex}[theorem]{Example}

\author[G. Bazzoni]{Giovanni Bazzoni}
\address{Fakult\"{a}t f\"{ur} Mathematik, Universit\"{a}t Bielefeld, Postfach 100301, D-33501 Bielefeld}

\email{gbazzoni@math.uni-bielefeld.de}

\author[J. C. Marrero]{Juan Carlos Marrero}
\address{Unidad Asociada ULL-CSIC, Geometr{\'\i}a Diferencial y Mec\'anica Geom\'etrica, Universidad de La Laguna, Facultad de Matem\'aticas, Avda. Astrof\'isico Francisco S\'anchez s/n. 38071, La Laguna}

\email{jcmarrer@ull.edu.es}

\title[Compact locally conformal symplectic nilmanifolds]{Locally conformal symplectic nilmanifolds with no locally conformal K\"ahler metrics}

\begin{document}

\begin{abstract}
We obtain an example of a compact locally conformal symplectic nilmanifold which admits no locally conformal K\"ahler metrics. This gives a new positive answer to a question raised by L. Ornea and M.
Verbitsky.
\end{abstract}

\maketitle

%%%%%%%%%%%%%%%%%%%%%%%%%%%%%%%%%%%%%%%%%%%%%%%%%%%%%%%%%%%%%%%%%%%%%
\section{Introduction}\label{sec:Introduction}
%%%%%%%%%%%%%%%%%%%%%%%%%%%%%%%%%%%%%%%%%%%%%%%%%%%%%%%%%%%%%%%%%%%%%

In the seminal paper \cite{GH}, Gray and Hervella studied almost Hermitian structures on even dimensional manifolds. Recall that an almost Hermitian structure on a manifold
$M$ of dimension $2n$ consists of a Riemannian metric $g$ and a compatible almost complex structure $J$, that is, an endomorphism $J\colon TM\to TM$ with $J^2=-\mathrm{Id}$, such
that $g(JX,JY)=g(X,Y)$ for every $X,Y\in\fX(M)$. Let $(g,J)$ be an almost Hermitian structure on $M$; then one can consider the following tensors on $M$:
\begin{itemize}
 \item a 2-form $\omega\in\Omega^2(M)$, the K\"ahler form, defined by $\omega(X,Y)=g(JX,Y)$, $X,Y\in\fX(M)$;
 \item a 1-form $\theta\in\Omega^1(M)$, the Lee form, defined by $\theta(X)=\frac{-1}{n-1}\delta\omega(JX)$, where $\delta$ is the codifferential and $X\in\fX(M)$.
\end{itemize}
Gray and Hervella classified almost Hermitian structures $(g,J)$ by studying the covariant derivative of the K\"ahler form $\omega$ with respect to the Levi-Civita connection $\nabla$ of $g$. Of
particular interest for us are the following classes:
\begin{itemize}
 \item the class of K\"ahler structures, for which $\nabla\omega=0$. In this case, $g$ is said to be a K\"ahler metric.  Compact examples of manifolds endowed with K\"ahler structures are Riemann
surfaces, complex tori, complex projective spaces and projective varieties; notice that the Lee form is zero on a K\"ahler manifold;
 %\item the class of almost K\"ahler manifolds, for which $d\omega=0$;
 %\item the class $\mathcal{C}$ of complex manifolds, for which $N_J=0$, where $N_J$ is the Nijenhuis tensor;
 \item the class of locally conformal K\"ahler (lcK from now on) structures, for which $d\omega=\omega\wedge\theta$; in this case we call $g$ a lcK metric. A remarkable example of manifold endowed
with a lcK structure is the Hopf surface. The study of lcK metrics on compact complex surfaces was undertaken in \cite{Be}.
\end{itemize}

In this note we will use indifferently the terminology (locally conformal) K\"ahler \emph{manifold} and (locally conformal) K\"ahler \emph{structure}. For a comprehensive introduction to K\"ahler
geometry we refer to \cite{Hu}; for lcK geometry, the standard reference is \cite{DO}; see also the recent survey \cite{OV}.

K\"ahler and locally conformal K\"ahler manifolds are examples of complex manifolds; this means that the almost complex structure $J$ is integrable, hence our manifold $M^{2n}$ is locally modelled on
$\bC^n$. Furthermore, on K\"ahler and locally conformal K\"ahler manifolds, any two of the tensors $(g,J,\omega)$ determine the third.
One can interpret K\"ahler manifolds as a ``degenerate'' case of locally conformal K\"ahler manifolds. Indeed, it turns out that the Lee form of a lcK structure is closed; suppose that it is exact; one can then
easily show that there is a K\"ahler metric in the conformal class of $g$. In such case, one says that the structure is globally conformal K\"ahler (gcK). In general, genuine lcK manifolds (those for
which the Lee form is not exact) admit an open cover $\{U_\alpha\}$ such that the Lee form is exact on each $U_\alpha$, hence the lcK metric is locally conformal to a K\"ahler metric.

The non-metric version of K\"ahler manifolds are symplectic manifolds (see \cite{McDS}). A manifold $M^{2n}$ is symplectic if there exists a 2-form $\omega\in\Omega^2(M)$ such that
$d\omega=0$
and $\omega^n$ is a volume form. Clearly, the K\"ahler form of a K\"ahler structure is symplectic. It is well known that the existence of a K\"ahler metric on a compact manifold deeply influences its
topology. In fact, suppose $M$ is a compact K\"ahler manifold of dimension $2n$ and let $\omega$ be the K\"ahler form; then:
\begin{itemize}
 \item the odd Betti numbers of $M$ are even;
 \item the Lefschetz map $H^{p}(M)\to H^{2n-p}(M)$, $[\alpha]\mapsto [\alpha\wedge\omega^{n-p}]$ is an isomorphism for $0\leq p\leq n$;
 \item $M$ is formal.
\end{itemize}
For a long time, the only known examples of symplectic manifolds came from K\"ahler (or even projective) geometry. In \cite{Th}, Thurston constructed the first example of a compact, symplectic
4-manifold which violates each of the three conditions given above. Since then, the problem of constructing compact symplectic manifolds without
K\"ahler structures has inspired beautiful Mathematics (see, for instance, \cite{FM,Go,McD,OT}).

The non-metric version of locally conformal K\"ahler manifolds are locally conformal symplectic (lcs) manifolds (see for instance \cite{Ba,Va1}). A manifold $M^{2n}$, endowed with
an almost symplectic structure $\omega$ (that is, $\omega^n$ is a volume form on $M$), is said to be lcs if there exists a closed 1-form $\theta$ such that $d\omega = \omega \wedge \theta$. When $n
\geq 2$ the $1$-form $\theta$ is uniquely determined and is called the Lee form. The defining condition of a lcs structure is equivalent  to the existence of an
open cover $\{U_\alpha\}$ of $M$ such that the Lee form is exact on each open set $U_\alpha$; on such $U_\alpha$, a suitable multiple $\omega'$ of the form $\omega$ is closed, hence $\omega'$
is a symplectic structure on $U_\alpha$. Analogous to the K\"ahler case, the K\"ahler form and the Lee form of a locally conformal K\"ahler structure define a locally conformal symplectic structure. A
lcs manifold is globally conformal symplectic (gcs) if the Lee form is exact; in this case, a multiple of $\omega$ is closed. A lcs manifold with vanishing first Betti number is automatically gcs.

The three properties of compact K\"ahler manifolds that we listed above constrain the topology of compact K\"ahler manifolds; on the other hand, the topology of symplectic manifolds is not
constrained, see for instance \cite{Ca}. In contrast, not much is known on the topology of compact lcK manifolds. In \cite[Conjecture 2.1]{DO} it was conjectured that a compact lcK manifold which
satisfies the topological restrictions of a K\"ahler manifold admits a (global) K\"ahler metric. A stronger conjecture (see \cite[Conjecture 2.2]{DO}) is that a compact lcK manifold
which is not gcK must have at least one Betti number which is odd. The latter conjecture was proven in the context of Vaisman structures, namely it was shown in \cite{KS,Va} that a compact Vaisman
manifold has odd $b_1$. Recall that Vaisman structures are a special class of lcK structures: a lcK metric $g$ is of Vaisman type if, in the conformal class of $g$, there is a metric $g'$ such that
the Lee form is parallel with respect to the Levi-Civita connection of $g'$. However, the statement was disproven for general lcK metrics; in \cite{OeTo}, Oeljeklaus and Toma constructed a lcK
manifold with $b_1=2$.
% According to the work of Vaisman, it was conjectured that a compact lcK (but not gcK) manifold $M$ of dimension $2n$ has a Betti number $b_{2p+1}$ which is odd, $0\leq p\leq
% \lfloor\frac{n-1}{2}\rfloor$.
% This conjecture was proven for lcK metrics of Vaisman type: $b_1$ is odd on manifolds endowed with such lcK metrics, see \cite{KS,Va}. 

As a consequence of the lack of topological obstructions, it is not easy to exclude that a manifold carries a locally conformal K\"ahler
metric. This motivates the following problem, proposed by Ornea and Verbitsky in \cite[Open Problem 1]{OV}:
\begin{quote}
 \emph{Construct a compact lcs manifold which admits no lcK metrics.}
\end{quote}

In \cite{BK}, Bande and Kotschick gave a method to construct examples of compact lcs manifolds which carry no lcK metrics; we describe it briefly. Suppose that $M$ is 
an oriented compact manifold of dimension $3$. Using some results on contact structures in dimension 3 \cite{Mar}, we have that $M$ admits a contact structure and, then, the product manifold $M
\times S^1$ is lcs (see \cite{BK,Va1}). On the other hand, using some results on compact complex surfaces, one deduces that $M$ may be chosen in such a way that the product manifold $M \times S^1$
admits no complex structures (for instance, if $M$ is hyperbolic) and, in particular, no lcK structures.

%  they give an example of a non simply connected symplectic and complex
%manifold $Z$ of real dimension 6 with $b_1(Z)=0$ which does not carry any K\"ahler metric. Indeed, the fundamental group $\Gamma$ of $Z$ is not the fundamental group of any
%K\"ahler manifold. Since $b_1(Z)=0$, however, this example belongs more to the realm of (globally conformal) symplectic and complex manifolds without any K\"ahler structure (see \cite{BM} for an
%account on the subject).

The main result of this note is to present another kind of examples of compact lcs manifolds which admit no lcK metrics. More precisely, we prove:

\begin{theorem}\label{main}
 There exists an example of a 4-dimensional compact nilmanifold which carries a locally conformal symplectic structure but no locally conformal K\"ahler metrics. This manifold is not a product of a
3-dimensional compact manifold and a circle.
\end{theorem}

The example of Thurston that we mentioned above (which admits a Vaisman structure) is also a nilmanifold. This makes the parallel between the symplectic and the locally conformal symplectic case particularly transparent.

This note is organized as follows. In Section \ref{sec:Nilmanifolds} we review some basics about nilmanifolds and describe briefly Thurston's example. In Section \ref{sec:Main} we prove the main
result.

{\bf Acknowledgements} J. C. Marrero acknowledges support from MEC (Spain) Grants
 MTM2011-15725-E, MTM2012-34478 and the project of the Canary Government ProdID20100210. G. Bazzoni is partially supported by MEC (Spain) Grant
 MTM2010-17389. The authors would like to thank Prof. M. Fern\'andez, L. Ornea and M. Verbitsky for useful comments and remarks.

%%%%%%%%%%%%%%%%%%%%%%%%%%%%%%%%%%%%%%%%%%%%%%%%%%%%%%%%%%%%%%%%%%%%%
\section{Basics on nilmanifolds}\label{sec:Nilmanifolds}
%%%%%%%%%%%%%%%%%%%%%%%%%%%%%%%%%%%%%%%%%%%%%%%%%%%%%%%%%%%%%%%%%%%%%

Let $G$ be a simply connected $n$-dimensional nilpotent Lie group. A \emph{nilmanifold} is a homogeneous space $N=\Gamma\backslash G$ where $\Gamma\subset G$ is a
lattice. Let $\fg$ be the Lie algebra of $G$; then $\fg$ is a nilpotent Lie algebra and the exponential map $\exp\colon\fg\to G$ is a global diffeomorphism. Hence $G$ is diffeomorphic to $\bR^n$. This
entails that $p\colon G\to N$ is a covering map and that $N$ is the Eilenberg-MacLane space $K(\Gamma,1)$. Tori are simple examples of nilmanifolds. A slightly less trivial example is the following:
let $H$ be the Heisenberg group, defined as
\[
H=\left\{\begin{pmatrix}
1 & y & z\\
0 & 1 & x\\
0 & 0 & 1
\end{pmatrix} \ | \ x,y,z \in\bR\right\}.
\]
Then $H$ is a nilpotent Lie group. A lattice $H_\bZ$ in $H$ is provided by matrices in $H$ such that $x,y,z\in\bZ$. Hence $H_\bZ\backslash H$ is a compact nilmanifold.

There is a simple criterion for the existence of a lattice $\Gamma$ in a nilpotent Lie group $G$:

\begin{theorem}[Mal'cev, \cite{Ma}]
 Let $G$ be a simply connected nilpotent Lie group with Lie algebra $\fg$. Then $G$ admits a lattice if and only if there exists a basis of $\fg$ with respect to which the structure constants are
rational numbers.
\end{theorem}

Suppose $N=\Gamma\backslash G$ is a compact nilmanifold and let us identify $\fg$ with left-invariant vector fields on $G$. These vector fields are, in particular, invariant under the (left) action of
$\Gamma$ on $G$. We can also identify $\fg^*$, the dual of $\fg$, with left-invariant 1-forms on $G$. In order to describe a geometric structure on a nilmanifold $N$ (given, for example, in terms of
forms and vector fields), we can specify its ``value'' at the (class of) the identity element $[e]\in N$, that is, by giving a corresponding geometric structure in $\fg$.

Let $\{X_1,\ldots,X_n\}$ be a basis of $\fg$ and let $\{x_1\ldots,x_n\}$ denote the dual basis of $\fg^*$, defined by $x_i(X_j)=\delta_{ij}$. Dualizing the bracket, we obtain a map
$d\colon\fg^*\to\wedge^2\fg^*$; for $x_i\in\fg^*$, $dx_i(X_j,X_k)=-x_i([X_j,X_k])$. $d$ is then extended to the whole exterior algebra $\wedge^*\fg^*$ by imposing
the graded Leibnitz rule. We thus obtain a differential complex $(\wedge^*\fg^*,d)$, called Chevalley-Eilenberg complex. The property $d^2=0$ is equivalent to the Jacobi identity in $\fg$. The
Chevalley-Eilenberg complex $(\wedge^*\fg^*,d)$ computes the Lie algebra cohomology of $\fg$.

\begin{ex}
 Let $N=\Gamma\backslash G$ be a nilmanifold of dimension $2n$. An invariant symplectic structure on $N$ is described by the mean of a 2-form $\omega\in\wedge^2\fg^*$ such that $d\omega=0$ and
$\omega^n\neq 0$.
\end{ex}

\begin{ex}\label{KT}
Let $G$ denote the product $H\times\bR$, where $H$ is the Heisenberg group described above. It is easy to see that the Lie algebra $\fg$ of $G$ admits a basis $\{X_1,\ldots,X_4\}$ with only non-zero
bracket $[X_1,X_2]=-[X_2,X_1]=-X_4$. Accordingly, $\fg^*$ has a basis $\{x_1,\ldots,x_4\}$ with $dx_1=dx_2=dx_3=0$ and $dx_4=x_1\wedge x_2$. Then $\omega=x_1\wedge x_4+x_2\wedge x_3\in\wedge^2\fg^*$ is
closed and satisfies $\omega^2\neq 0$. Let $\Gamma\subset G$ denote the lattice $H_\bZ\times\bZ$ and let $N=\Gamma\backslash G$ be the corresponding nilmanifold. By abuse of notation, we denote also
by $\omega$ the 2-form on $N$ obtained from $\omega\in\wedge^2\fg^*$. Then $\omega\in\Omega^2(N)$ is closed and nondegenerate, hence $(N,\omega)$ is a compact symplectic nilmanifold.
\end{ex}

One feature of nilmanifolds, which makes them particularly tractable from a cohomological point of view, is that their de Rham cohomology is computed by the Chevalley-Eilenberg complex. More
precisely,

\begin{theorem}[Nomizu, \cite{No}]\label{Nomizu}
 Let $N=\Gamma\backslash G$ be a compact nilmanifold. Then the natural inclusion $(\wedge^*\fg^*,d)\hookrightarrow (\Omega^*(N),d)$ induces an isomorphism in cohomology.
\end{theorem}

\begin{corollary}\label{odd_Betti}
 Let $N=\Gamma\backslash G$ be a compact nilmanifold of dimension $n$. Then, for every $0\leq i\leq n$, $b_i(N)=b_i(\fg)$.
\end{corollary}

If we apply Corollary \ref{odd_Betti} to the nilmanifold $N$ of Example \ref{KT}, we see that $b_1(N)=3$, since 3 of the 4 generators of $\fg^*$ are closed. Comparing this
with the topological restrictions imposed by the existence of a K\"ahler metric on a compact manifold, we see immediately that $N$ can not carry any K\"ahler metric. 
One can also easily show that $N$ is neither Hard Lefschetz nor formal. The nilmanifold $N$ of Example \ref{KT} is known as the Kodaira-Thurston manifold. On the one hand, it appeared in the
classification of compact complex surfaces carried out by Kodaira (see \cite{Ko}); on the other hand, it was the first example of a compact
symplectic manifold without any K\"ahler structure in the celebrated paper \cite{Th} by Thurston.

As a matter of fact, compact K\"ahler nilmanifolds are very rare. More precisely, we have the following theorem:

\begin{theorem}[Benson-Gordon, \cite{BG} and Hasegawa, \cite{Ha}]\label{BG-HAS}
 Let $N$ be a $2n$-dimensional compact nilmanifold endowed with a K\"ahler structure. Then $N$ is diffeomorphic to the torus $T^{2n}$.
\end{theorem}

In particular, Benson and Gordon proved that a compact symplectic nilmanifold which satisfies the Hard Lefeschetz property must be diffeomorphic to a torus. Hasegawa showed that the same conclusion holds for
a formal compact nilmanifold.

As a consequence, compact symplectic non-toral nilmanifolds are never K\"ahler. The problem of determining which even-dimensional nilmanifolds carry a symplectic structure is still open (see for instance
\cite{Gu}), but is completely solved in dimension 2, 4 and 6 (see \cite{BaMu}). Notice that, by Nomizu Theorem, a compact nilmanifold has a symplectic structure if and only if it has a
left-invariant one.

In what follows, we shall adopt the standard notation for nilpotent Lie algebras/nilmanifolds. It is best explained by the mean of an example: $(0,0,0,0,12,34)$
denotes a 6-dimensional nilmanifold $N=\Gamma\backslash G$ such that $\fg$ has a basis $\{X_1,\ldots,X_6\}$ with dual basis $\{x_1,\ldots,x_6\}$ such that $dx_1=\ldots=dx_4=0$, $dx_5=x_1\wedge x_2$
and $dx_6=x_3\wedge x_4$, where $d$ is the Chevalley-Eilenberg differential. Different choices of lattices in $G$ give rise to coverings between the corresponding nilmanifolds. 

We have the following result
\begin{proposition}[see for instance \cite{BM,OT}]\label{4-dim-nilm}
 Let $N=\Gamma\backslash G$ be a compact 4-dimensional nilmanifold. Then $\fg^*$ is one of the following:
 \begin{enumerate}
  \item[1.] $(0,0,0,0)$ and $N$ is diffeomorphic to a torus $T^4$
  \item[2.] $(0,0,0,12)$ and $N$ is diffeomorphic to (a finite quotient of) the Kodaira-Thurston manifold
  \item[3.] $(0,0,12,13)$ and $N$ is diffeomorphic to (a finite quotient of) an iterated circle bundle over $T^2$.
 \end{enumerate}
Correspondingly, in each cases $N$ carries a (left-invariant) symplectic structure, which we describe in terms of a basis $\{x_1,\ldots,x_4\}$ of $\fg^*$;
 \begin{enumerate}
  \item[1.] $\omega=x_1\wedge x_2+x_3\wedge x_4$
  \item[2.] $\omega=x_1\wedge x_4+x_2\wedge x_3$
  \item[3.] $\omega=x_1\wedge x_4+x_2\wedge x_3$.
 \end{enumerate}
\end{proposition}

%%%%%%%%%%%%%%%%%%%%%%%%%%%%%%%%%%%%%%%%%%%%%%%%%%%%%%%%%%%%%%%%%%%%%
\section{Proof of the main theorem}\label{sec:Main}
%%%%%%%%%%%%%%%%%%%%%%%%%%%%%%%%%%%%%%%%%%%%%%%%%%%%%%%%%%%%%%%%%%%%%

We prove next two propositions, of which Theorem \ref{main} is a corollary.

\begin{proposition}\label{main:prop1}
 Let $N=\Gamma\backslash G$ be the nilmanifold $(0,0,12,13)$. Then $N$ is symplectic and locally conformal symplectic but does not admit any complex structure.
\end{proposition}
\begin{proof}
 A symplectic structure on $N$ has been described explicitly in Proposition \ref{4-dim-nilm}. In terms of the basis of $\fg^*$ considered above, the locally conformal symplectic structure is
described by taking the almost symplectic $2$-form to be $\omega=x_1\wedge x_3+x_4\wedge x_2$ and the Lee form to be $\theta=x_2$. Then $d\omega=\theta\wedge\omega$. Notice that $\theta$ is not exact, hence the
structure is genuinely lcs. Nomizu Theorem implies that $b_1(N)=2$. Assume that $N$ admits a complex structure. By standard results in the theory of compact complex surfaces (see for
instance \cite[Theorem 3.1, page 144]{BPVdV}), a complex surface with even first Betti number admits a K\"ahler metric. Since $N$ is clearly not diffeomorphic to $T^4$, this would give a
contradiction with Theorem \ref{BG-HAS}. Hence $N$ does not carry any complex structure.
\end{proof}

\begin{proposition}\label{main:prop2}
 The nilmanifold $N$ considered above is not the product of a compact 3-dimensional manifold and a circle.
\end{proposition}
\begin{proof}
 By Nomizu theorem, we see that $b_1(N)=2$. Assume $N$ is a product $M\times S^1$, where $M$ is a compact 3-dimensional manifold. Since $N$ is a compact nilmanifold, the fundamental group
$\Gamma$ of $N$ is nilpotent and torsion-free. If $\Lambda=\pi_1(M)$, then $\Lambda\subset\Gamma$ is also nilpotent and torsion-free. In \cite{Ma}, Mal'cev showed that for such a group $\Lambda$ there
exists a real nilpotent Lie
group $\Lambda_\bR$ such that $P=\Lambda\backslash\Lambda_\bR$ is a nilmanifold. It is well known that a compact nilmanifold has first Betti number at least 2, hence $b_1(P)\geq 2$. Now $P$ is an
aspherical manifold; in this case, $H^*(P;\bZ)$ is isomorphic to the group cohomology $H^*(\Lambda;\bZ)$ of $\Lambda$ with coefficients in the trivial $\Lambda$-module $\bZ$ (see for instance
\cite[page 40]{Br}). Hence $b_1(\Lambda)\geq 2$. Now $M$ is also an aspherical space with fundamental group $\Lambda$, hence, by the same token, $b_1(M)\geq 2$. However, by the K\"unnet formula
applied to $N=M\times S^1$, we get $b_1(M)=1$. Alternatively, we could apply \cite[Theorem 3]{FHT} directly to our manifold $M$.
\end{proof}

\begin{proof}[Proof of Theorem \ref{main}]
Assume $N$ carries a lcK metric. Then $N$ must be in particular a complex manifold, and this contradicts Proposition \ref{main:prop1}. By Proposition \ref{main:prop2}, $N$ is not a product of a
compact 3-manifold and a circle.
\end{proof}

The lcs structure of the compact nilmanifold $(0,0,12,13)$, which was considered in the proof of Proposition \ref{main:prop1}, is of the first kind in the terminology of Vaisman \cite{Va1}. A lcs manifold of the first kind is locally the product of a contact manifold with the real line \cite{Va1} and the typical example of a compact lcs manifold of the first kind is the product of a compact contact manifold with the circle $S^1$. However, our compact example is not globally a product manifold. This is a good motivation in order to discuss the global structure of lcs manifolds of the first kind. In fact, this topic is the subject of a work in progress \cite{BaMa}.

An explicit realization of the compact nilmanifold $(0,0,12,13)$ is the following one. Let $G$ be the connected simply connected Lie group $G = \bR^4$ endowed with the multiplication
\[
(x,y,z,t) \cdot (x',y',z',t') = (x + x', y + y', z + z' + yx', t + t' + zx' + y \frac{(x')^2}{2}).
\]  
Then, $(\bR^4, \cdot)$ is a nilpotent Lie group and 
\[
\{x_1 = dx, \; x_2 = dy, \; x_3 = dz - y dx, \; x_4 = dt - zdx\}
\]
is a basis of left-invariant $1$-forms on $G$. Moreover, 
\[ 
\Gamma = 2\bZ \times \bZ \times \bZ \times \bZ
\]
is a co-compact discrete subgroup of $G$. 

 The proof of Theorem \ref{main} relies heavily on the classification of compact complex surfaces. It is therefore unlikely that our argument can be extended to the higher dimensional case.

It was conjectured in \cite[Conjecture 5.3]{BK} that any compact complex surface also admits a lcK structure (not necessarily compatible with the given complex structure). By what we argued so far, 
this conjecture is true in the context of 4-dimensional compact nilmanifolds. The torus $T^4$ admits a (globally) conformally K\"ahler structure and the Kodaira-Thurston carries a Vaisman structure
(see \cite{Sa}). The nilmanifold $N$ which we used in the proof of Theorem \ref{main} is not complex.

4-dimensional nilmanifolds display interesting behaviors with respect to the geometric structure we are dealing with:
 \begin{itemize}
  \item the torus $T^4$ is K\"ahler;
  \item the Kodaira-Thurston manifold is complex, symplectic, Vaisman but not K\"ahler;
  \item the nilmanifold $(0,0,12,13)$ is symplectic, locally conformal symplectic, but not complex (hence neither K\"ahler nor lcK).
 \end{itemize}
 
 \begin{remark}
 It would be interesting to find an example of a manifold which is genuinely lcs and complex but neither symplectic nor lcK. It is not hard to come up with such an example if one puts the more
restrictive condition of not being of Vaisman type.
\end{remark}

 It was conjectured in \cite{Ug} that if a $2n$-dimensional compact nilmanifold $N=\Gamma\backslash G$ carries a lcK metric, then $\fg=\fh_{2n-1}\oplus\bR$, where $\fh_{2n-1}$ is the
generalized Heisenberg algebra; this is the Lie algebra with basis $\{X_1,Y_1,\ldots,X_{n-1},Y_{n-1},Z\}$ and non-zero brackets $[X_i,Y_i]=-[Y_i,X_i]=-Z$ for
$i=1,\ldots,n-1$ (notice that $\fh_3$ is the Lie algebra of the Heisenberg group). This conjecture has been proved by Sawai in \cite{Sa}, assuming that the complex structure of the lcK metric is
left-invariant. A positive answer to this conjecture would produce immediately many examples of lcs non lcK manifolds. Vaisman structures on solvmanifolds have been considered in \cite{Ka}. More
generally, it was proven recently in \cite{GMO} that compact homogeneous lcK structures are Vaisman (see also \cite{ACHK} for a non-compact extension of this result).

\begin{remark}
 Since the Kodaira-Thurston manifold carries a Vaisman metric, it is clear that compact locally conformal K\"ahler manifolds need not be formal. Can one, in some sense to be made precise, ``bound''
the non-formality of a compact lcK manifold (for instance, in terms of non-zero Massey products)? The study of geometric formality in the context of Vaisman metrics was tackled in \cite{OP}.
\end{remark}

% 
%  \bibliographystyle{plain}
%  \bibliography{papers,books}
%  
%  \end{document}
%  

\end{document}